\documentclass[preprint,12pt]{elsarticle}
\usepackage{amsfonts}

 \usepackage{graphics}
\usepackage{graphicx}
 \usepackage{epsfig}

\usepackage{amssymb}
\usepackage{amsthm}

\newtheorem{thm}{Theorem}
\newtheorem{lem}[thm]{Lemma}
\newdefinition{rmk}{Remark}
\newproof{pf}{Proof}
\newdefinition{example}{Example}
\newdefinition{definition}{Definition}
\newdefinition{proposition}{Proposition}
\newdefinition{problem}{Problem}
\newdefinition{corollary}{Corollary}

\journal{Arkiv}

\begin{document}

\begin{frontmatter}



\title{An $S$-type eigenvalue localization set for tensors}

\author[rvt]{Chaoqian Li}
\ead{lichaoqian@ynu.edu.cn}

\author[rvt]{Aiquan Jiao}
\ead{jaq1029@163.com}
\author[rvt]{Yaotang Li\corref{cor1}}
\ead{liyaotang@ynu.edu.cn}

\cortext[cor1]{Corresponding author.}


\address[rvt]{School of Mathematics and Statistics, Yunnan
University, Yunnan,  P. R. China 650091}

\begin{abstract}
An $S$-type eigenvalue localization set for a tensor is given by
breaking $N=\{1,2,\cdots,n\}$ into disjoint subsets $S$ and its
complement. It is shown that the new set is tighter than those
provided by L. Qi (Journal of Symbolic Computation 40 (2005)
1302-1324) and Li et al. (Numer. Linear Algebra Appl. 21 (2014)
39-50). As applications of the results, a checkable sufficient
condition for the positive definiteness of tensors and a checkable
sufficient condition of the positive semi-definiteness of tensors
are given.

\end{abstract}


\end{frontmatter}


\section{Introduction}
Eigenvalue problems of tensors have become an important topic of
study in numerical multilinear algebra, and they have a wide range
of practical applications; see
\cite{Bo,Ch,Ch1,Di,Hu,Ju,Ko1,Ku,Li,Liu,Ng,Ni,Qi1,Qi5,Wa,Wa1}. Here
we call $\mathcal {A}=(a_{i_1\cdots i_m})$ a complex (real) tensor
of order $m$ dimension $n$, denoted by $\mathcal {A}\in C^{[m,n]}~
(R^{[m,n]})$, if
\[a_{i_1\cdots i_m}\in C~ (R),\] where $i_j=1,\ldots,n$ for
$j=1,\ldots, m$. Moreover, if there are a complex number $\lambda$
and a nonzero complex vector $x=(x_1,x_{2},\ldots,x_n)^T$ such that
\[\mathcal {A}x^{m-1}=\lambda x^{[m-1]},\]
then $\lambda$ is called an eigenvalue of $\mathcal {A}$ and $x$ an
eigenvector of $\mathcal {A}$ associated with $\lambda$
\cite{Lim,Qi}, where $\mathcal {A}x^{m-1}$ is an $n$ dimension
vector whose $i$th component is
\[(\mathcal {A}x^{m-1})_i=\sum\limits_{i_2,\ldots,i_m\in N}
a_{ii_2\cdots i_m}x_{i_2}\cdots x_{i_m} ~(N =\{1,2,\ldots,n\})\] and
\[ x^{[m-1]}=(x_1^{m-1},x_2^{m-1},\ldots,x_n^{m-1})^T.\]
 If $x$ and $\lambda$ are all real, then
$\lambda$ is called an H-eigenvalue of $\mathcal {A}$ and $x$ an
H-eigenvector of $\mathcal {A}$ associated with $\lambda$
\cite{Qi,Qi1,Ya1}.

One of many practical applications of eigenvalues of tensors is that
one can identify the positive (semi-)definiteness for an even-order
real symmetric tensor by using the smallest H-eigenvalue of a
tensor, consequently, can identify the positive (semi-)definiteness
of the multivariate homogeneous polynomial determined by this
tensor, for details, see \cite{Li,Qi}.

Because it is not easy to compute the smallest H-eigenvalue of
tensors when the order and dimension are  large, ones always try to
give a set including all eigenvalues in the complex
\cite{Qi,Li1,Li2,Li3}. In particular, if this set for an even-order
real symmetric tensor is in the right-half complex plane, then we
can conclude that the smallest H-eigenvalue is positive,
consequently, the corresponding tensor is positive definite.

In \cite{Qi}, Qi gave an eigenvalue localization set for  real
symmetric tensors, which is a generalization of the well-known
Ger\v{s}gorin's eigenvalue localization set of matrices
\cite{Ger,Va}. This result can be easily generalized to general
tensors \cite{Li1,Ya}.

\begin{thm}\cite{Li1,Qi,Ya}
\label{Qi-Li} Let $\mathcal {A}=(a_{i_1\cdots i_m})\in C^{[m,n]}$.
Then
\[
\sigma(\mathcal {A})\subseteq \Gamma(\mathcal
{A}):=\bigcup\limits_{i\in N} \Gamma_i(\mathcal {A}),
\]
where $\sigma (\mathcal {A})$ is the set of all the eigenvalues of
$\mathcal{A}$, \[\Gamma_i(\mathcal {A})=\left\{z\in
\mathbb{C}:|z-a_{i\cdots i}|\leq r_i(\mathcal
{A})\right\},~r_i(\mathcal {A})= \sum\limits_{i_2,\ldots,i_m\in
N,\atop \delta_{ii_2\ldots i_m}=0} |a_{ii_2\cdots i_m}|\] and
\[\delta_{i_1\cdots i_m}=\left\{\begin{array}{cl}
   1,   &if~ i_1=\cdots =i_m,  \\
   0,   &otherwise.
\end{array}
\right.\]
\end{thm}

Although it is easy to get $\Gamma(\mathcal {A})$ in the complex by
computing $n$ sets $\Gamma_i(\mathcal {A})$, $\Gamma(\mathcal {A})$
does'nt always capture all eigenvalues of $\mathcal {A}$ very
precisely. To obtain tighter sets than $\Gamma(\mathcal {A})$, Li et
al. \cite{Li1} extended the Brauer's eigenvalue localization set of
matrices \cite{brauer,Va} and gave the following Brauer-type
eigenvalue localization set for tensors.

\begin{thm}\cite{Li1}
\label{Li-Li}
 Let $\mathcal {A}=(a_{i_1\cdots i_m})\in C^{[m,n]}$, $n\geq 2$. Then
\[
\sigma(\mathcal {A})\subseteq \mathcal{K}(\mathcal {A}):=
\bigcup\limits_{i,j\in N,\atop j\neq i} \mathcal {K}_{i,j}(\mathcal
{A}),
\]
where
\[\mathcal {K}_{i,j}(\mathcal {A})=\left\{z\in
\mathbb{C}:\left(|z-a_{i\cdots i}|-r_i^j(\mathcal
{A})\right)|z-a_{j\cdots j}|\leq |a_{ij\cdots j}|r_j(\mathcal
{A})\right\}\] and
\[r_i^j(\mathcal
{A})=\sum\limits_{\delta_{i,i_2,\ldots, i_m}=0,\atop
\delta_{j,i_2,\ldots, i_m}=0} |a_{ii_2\cdots i_m}|=r_i(\mathcal
{A})-|a_{ij\cdots j}|.\] Furthermore, $\mathcal{K}(\mathcal {A})
\subseteq \Gamma(\mathcal {A})$.
\end{thm}

As Theorem \ref{Li-Li} shows, we need compute $n(n-1)$ sets
$\mathcal {K}_{i,j}(\mathcal {A})$ to give the set
$\mathcal{K}(\mathcal {A})$, however $\mathcal{K}(\mathcal {A})$
captures all eigenvalues of $\mathcal {A}$ more precisely than
$\Gamma(\mathcal {A})$. To reduce computations,  Li et al. give an
$S$-type eigenvalue localization set by breaking $N$ into disjoint
subsets $S$ and $\bar{S}$, where $\bar{S}$ is the complement of $S$
in $N$.

\begin{thm} \cite[Theorem 2.2]{Li1}
\label{Li-Li1} Let $\mathcal {A}=(a_{i_1\cdots i_m})\in C^{[m,n]}$,
$n\geq 2$, and $S$ be a nonempty proper subset of $N$. Then
\[
\sigma(\mathcal {A})\subseteq \mathcal{K}^S (\mathcal
{A}):=\left(\bigcup\limits_{i\in S,\atop j\neq \bar{S}} \mathcal
{K}_{i,j}(\mathcal {A})\right) \bigcup \left( \bigcup\limits_{i\in
\bar{S},\atop j\neq S} \mathcal {K}_{i,j}(\mathcal {A})\right).
\]
\end{thm}

The set $\mathcal{K}^S (\mathcal {A})$ in Theorem \ref{Li-Li1}
consists of $2|S|(n-|\bar{S}|)$ sets $\mathcal {K}_{i,j}(\mathcal
{A})$, where $|S|$ is the cardinality of $S$. It is obvious that
$2|S|(n-|\bar{S}|) \leq n(n-1),$ and then \begin{equation}
\label{com1} \mathcal{K}^S (\mathcal {A}) \subseteq
\mathcal{K}(\mathcal {A}) \subseteq \Gamma(\mathcal
{A}),\end{equation} for details, see \cite{Li1}. In this paper, by
the technique in \cite{Li1} we  give a new eigenvalue localization
set involved with a proper subset $S$ of $N$, and prove that the new
set  is tighter than $\Gamma(\mathcal {A}), \mathcal{K}(\mathcal
{A})$ and $\mathcal{K}^S (\mathcal {A})$. As an application, we give
some checkable sufficient conditions for the positive
(semi-)definiteness of tensors.

\section{A new $S$-type eigenvalue localization set}\vspace{-2pt}
we begin with some notation. Given an nonempty proper subset $S$ of
$N$, we denote
\[ \Delta^N :=\{(i_2,i_3,\cdots,i_m): each ~i_j\in N ~for ~j=2,\cdots, m \},\]
\[ \Delta^S :=\{(i_2,i_3,\cdots,i_m): each ~i_j\in S ~for~ j=2,\cdots, m \},\]
and then
\[ \overline{\Delta^S}= \Delta^N  \backslash \Delta^S.\]
This implies that for a tensor $\mathcal {A}=(a_{i_1\cdots i_m})\in
C^{[m,n]},$ we have that for $i\in S$, \[r_i(\mathcal {A})=
r_i^{\Delta^S}(\mathcal {A})+r_i^{\overline{\Delta^S}}(\mathcal
{A}),~r_i^j(\mathcal {A})= r_i^{\Delta^S}(\mathcal
{A})+r_i^{\overline{\Delta^S}}(\mathcal {A})-|a_{ij\cdots
j}|,\]where \[ r_i^{\Delta^S}(\mathcal
{A})=\sum\limits_{(i_2,\cdots,i_m)\in \Delta^S,
 \atop \delta_{ii_2\cdots i_m}=0} |a_{ii_2\cdots i_m}|,
 r_i^{\overline{\Delta^S}}(\mathcal {A})=\sum\limits_{(i_2,\cdots,i_m)\in \overline{\Delta^S}} |a_{ii_2\cdots
 i_m}|\]

\begin{thm}
\label{nth2.1} Let $\mathcal {A}=(a_{i_1\cdots i_m})\in C^{[m,n]}$,
$n\geq 2$, and $S$ be a nonempty proper subset of $N$. Then
\[
\sigma(\mathcal {A})\subseteq \Omega^S(\mathcal {A}):=\left(
\bigcup\limits_{i\in S,\atop j\in \bar{S}} \Omega_{i,j}^S(\mathcal
{A}) \right)\bigcup\left( \bigcup\limits_{i\in \bar{S},\atop j\in S}
\Omega_{i,j}^{\bar{S}}(\mathcal {A}) \right),\] where
\[\Omega_{i,j}^S(\mathcal
{A}) :=\left\{z\in \mathbb{C}:\left( |\lambda-a_{i\cdots i}|\right)
\left(|\lambda-a_{j\cdots j}|-r_j^{\overline{\Delta^{S}}}(\mathcal
{A})\right)\leq r_i(\mathcal {A}) r_j^{\Delta^{S}}(\mathcal {A})
\right\}\] and
\[\Omega_{i,j}^{\bar{S}}(\mathcal
{A}) :=\left\{z\in \mathbb{C}:\left( |\lambda-a_{i\cdots i}|\right)
\left(|\lambda-a_{j\cdots
j}|-r_j^{\overline{\Delta^{\bar{S}}}}(\mathcal {A})\right)\leq
r_i(\mathcal {A}) r_j^{\Delta^{\bar{S}}}(\mathcal {A}) \right\}.\]
\end{thm}

\begin{proof} For any $\lambda\in \sigma(\mathcal {A})$, let $ x=(x_1,x_2,\ldots,x_n)^T\in
\mathbb{C}^n\backslash \{0\}$ be an associated eigenvector, i.e.,
\begin{equation}\label{eq 2.2}
\mathcal {A}x^{m-1}=\lambda x^{[m-1]}. \end{equation} Let
$|x_p|=\max\limits_{i\in S} |x_i| $ and $|x_q|=\max\limits_{i\in
\bar{S}} |x_i| $ . Then, at least one of $x_p$ and $x_q$ is nonzero.
We next divide into three cases to prove.

 Case I: $x_p x_q\neq 0$ and $|x_q|\geq |x_p|$, that is,
$|x_q|=\max\limits_{i\in N} |x_i| $.  By (\ref{eq 2.2}), we have
\[(\lambda-a_{q\cdots q})
x_q^{m-1}=\sum\limits_{(i_2,\cdots,i_m)\in \Delta^{S}} a_{qi_2\cdots
i_m}x_{i_2}\cdots x_{i_m}+ \sum\limits_{(i_2,\cdots,i_m)\in
\overline{\Delta^{S}}, \atop \delta_{qi_2\cdots i_m}=0}
a_{qi_2\cdots i_m}x_{i_2}\cdots x_{i_m}.
\] Taking modulus in the above equation and using the triangle
inequality gives
\begin{eqnarray*}
|\lambda-a_{q\cdots q}||x_q|^{m-1}&\leq &
\sum\limits_{(i_2,\cdots,i_m)\in \Delta^{S}} |a_{qi_2\cdots
i_m}||x_{i_2}|\cdots |x_{i_m}|+ \sum\limits_{(i_2,\cdots,i_m)\in
\overline{\Delta^{S}}, \atop
\delta_{qi_2\cdots i_m}=0} |a_{qi_2\cdots i_m}||x_{i_2}|\cdots |x_{i_m}|\nonumber\\
&\leq& \sum\limits_{(i_2,\cdots,i_m)\in \Delta^{S}} |a_{qi_2\cdots
i_m}||x_p|^{m-1}+ \sum\limits_{(i_2,\cdots,i_m)\in
\overline{\Delta^{S}}, \atop
\delta_{qi_2\cdots i_m}=0} |a_{qi_2\cdots i_m}||x_q|^{m-1}\\
&=& r_q^{\Delta^{S}}(\mathcal
{A})|x_p|^{m-1}+r_q^{\overline{\Delta^{S}}}(\mathcal
{A})|x_q|^{m-1}.\nonumber
\end{eqnarray*}
Hence, \begin{equation}\label{eq 3.2} \left( |\lambda-a_{q\cdots
q}|-r_q^{\overline{\Delta^{S}}}(\mathcal {A})\right)|x_q|^{m-1}\leq
r_q^{\Delta^{S}}(\mathcal {A})|x_p|^{m-1}.
\end{equation}
On the other hand, by (\ref{eq 2.2}), we also get that
\[(\lambda-a_{p\cdots p})x_p^{m-1}=\sum\limits_{i_2,\cdots,i_m\in N ,  \atop
\delta_{pi_2\cdots i_m}=0} a_{pi_2\cdots i_m} x_{i_2}\cdots
x_{i_m}\] and
\begin{equation}\label{eq3.2n} |\lambda-a_{p\cdots p}||x_p|^{m-1}\leq r_p(\mathcal
{A})|x_q|^{m-1}.
\end{equation}
Multiplying (\ref{eq 3.2}) with (\ref{eq3.2n}) gives
 \[ ( |\lambda-a_{p\cdots
p}|) \left( |\lambda-a_{q\cdots
q}|-r_q^{\overline{\Delta^{S}}}(\mathcal {A})\right)\leq
r_p(\mathcal {A}) r_q^{\Delta^{S}}(\mathcal {A}),\] which leads to
$\lambda \in  \Omega^{S}_{p,q}(\mathcal {A})\subseteq
\Omega^S(\mathcal {A}).$

Case II:  $x_p x_q\neq 0$ and $|x_p|\geq |x_q|$, that is,
$|x_p|=\max\limits_{i\in N} |x_i| $. Similar to the proof of Case I,
we can obtain that
\[  \left( |\lambda-a_{p\cdots
p}|-r_p^{\overline{\Delta^{\bar{S}}}}(\mathcal
{A})\right)|x_p|^{m-1}\leq r_p^{\Delta^{\bar{S}}}(\mathcal
{A})|x_q|^{m-1},
\] and
\[ |\lambda-a_{q\cdots
q}||x_q|^{m-1}\leq r_q(\mathcal {A})|x_p|^{m-1}.\] This gives
\[ \left( |\lambda-a_{q\cdots
q}| \right) \left( |\lambda-a_{p\cdots
p}|-r_p^{\overline{\Delta^{\bar{S}}}}(\mathcal {A})\right)\leq
 r_q(\mathcal
{A}) r_p^{\Delta^{\bar{S}}}(\mathcal {A}). \] Hence, $\lambda \in
\Omega^{\bar{S}}_{q,p}(\mathcal {A}) \subseteq \Omega^S(\mathcal
{A}).$

Case III: $|x_p||x_q|= 0$, without loss of generality, let $|x_p|=0
$ and $|x_q |\neq 0$. Then by (\ref{eq 3.2}),
\[|\lambda-a_{q\cdots
q}|-r_q^{\overline{\Delta^{S}}}(\mathcal {A})\leq 0.\] hence for any
$i\in S$,
\[ \left( |\lambda-a_{i\cdots
i}|\right) \left( |\lambda-a_{q\cdots
q}|-r_q^{\overline{\Delta^{S}}}(\mathcal {A}) \right)\leq
r_i(\mathcal {A}) r_q^{\Delta^{S}}(\mathcal {A}),\] which leads to
$\lambda \in \Omega^{S}_{i,q}(\mathcal {A})\subseteq
\Omega^S(\mathcal {A}).$ The conclusion follows from Cases I, II and
III. \end{proof}

To compare the sets $\Gamma(\mathcal {A})$ in Theorem \ref{Qi-Li},
$\mathcal{K}(\mathcal {A})$ in Theorem \ref{Li-Li},
$\mathcal{K}^{S}(\mathcal {A})$ in Theorem \ref{Li-Li1} with
$\Omega^S(\mathcal {A})$ in Theorem \ref{nth2.1}, two lemmas in
\cite{Li2} are listed as follows.

\begin{lem}\cite[Lemmas 2.2 and 2.3 ]{Li2} \label{nle2.1}
(I) Let $a,b,c\geq 0$ and $d>0$. If $ \frac{a}{b+c+d}\leq 1$, then
\[ \frac{a-(b+c)}{d} \leq \frac{a-b}{c+d} \leq \frac{a}{b+c+d}.\]

(II) Let $a,b,c\geq 0$ and $d>0$. If $ \frac{a}{b+c+d}\geq 1$, then
\[ \frac{a-(b+c)}{d} \geq \frac{a-b}{c+d} \geq \frac{a}{b+c+d}.\]
\end{lem}

\begin{thm} \label{com}
Let $\mathcal {A}=(a_{i_1\cdots i_m})\in C^{[m,n]}$, $n\geq 2$. And
let $S$ be a nonempty proper subset of $N$. Then
\[\Omega^S(\mathcal {A}) \subseteq \mathcal{K}^S(\mathcal {A})
\subseteq \mathcal{K}(\mathcal {A}) \subseteq \Gamma(\mathcal{A}).\]
\end{thm}

\begin{proof} By (\ref{com1}), $\mathcal{K}^S(\mathcal {A})
\subseteq \mathcal{K}(\mathcal {A}) \subseteq \Gamma(\mathcal{A})$
holds.
 We only prove $\Omega^S(\mathcal {A}) \subseteq \mathcal{K}^S(\mathcal {A})$.
Let $z\in \Omega^S(\mathcal {A})$. Then
\[z\in \bigcup\limits_{i\in S,\atop j\in \bar{S}}
\Omega_{i,j}^S(\mathcal {A}) ~or~ z\in \bigcup\limits_{i\in
\bar{S},\atop j\in S} \Omega_{i,j}^{\bar{S}}(\mathcal {A}).\]
Without loss of generality, suppose that  $z\in \bigcup\limits_{i\in
S,\atop j\in \bar{S}} \Omega_{i,j}^S(\mathcal {A})$ (we can prove it
similarly if $z\in \bigcup\limits_{i\in \bar{S},\atop j\in S}
\Omega_{i,j}^{\bar{S}}(\mathcal {A})$ ). Then there are $p\in S$ and
$q\in \bar{S}$ such that $z\in \Omega_{p,q}^S(\mathcal {A})$, i.e.,
 \begin{equation} \label{neq2.3} \left( |z-a_{p\cdots p}|\right)
\left(|z-a_{q\cdots q}|-r_q^{\overline{\Delta^{S}}}(\mathcal
{A})\right)\leq r_p(\mathcal {A}) r_q^{\Delta^{S}}(\mathcal
{A}).\end{equation}

If $r_p(\mathcal {A}) r_q^{\Delta^{S}}(\mathcal {A})=0$, then
$r_p(\mathcal {A})=0,$ or $ r_q^{\Delta^{S}}(\mathcal {A})=0$. When
$r_q^{\Delta^{S}}(\mathcal {A})=0,$ we have $ |a_{qp\cdots p}|=0$,
$r_q^{\Delta^{S}}(\mathcal {A})= r_q^p(\mathcal {A})$ and
\begin{eqnarray*}|z-a_{p\cdots
p}|\left(|z-a_{q\cdots q}|-r_q^p(\mathcal {A})\right)&=&
|z-a_{p\cdots p}|\left(|z-a_{q\cdots
q}|-r_q^{\overline{\Delta^{S}}}(\mathcal {A})\right)\\&\leq&
r_p(\mathcal {A}) r_q^{\Delta^{S}}(\mathcal {A})\\&=& r_p(\mathcal
{A})|a_{qp\cdots p}| \\&=&0,\end{eqnarray*}which implies that $ z\in
\mathcal{K}_{q,p}(\mathcal {A}) \subseteq \bigcup\limits_{i\in
\bar{S},\atop j\in S} \mathcal{K}_{i,j}(\mathcal {A}) \subseteq
\mathcal{K}^S(\mathcal {A}),$ consequently, $\Omega^S(\mathcal {A})
\subseteq \mathcal{K}^S(\mathcal {A})$. When $r_p(\mathcal {A})=0$,
we have
\begin{eqnarray*}
|z-a_{p\cdots p}|\left(|z-a_{q\cdots q}|-r_q^p(\mathcal {A})\right)
&\leq&
|z-a_{p\cdots p}|\left(|z-a_{q\cdots q}|-r_q^{\overline{\Delta^{S}}}(\mathcal {A})\right)\\
&\leq&  r_p(\mathcal {A}) r_q^{\Delta^{S}}(\mathcal
{A})\\
&=&0\\
&=&r_p(\mathcal {A})|a_{qp\cdots p}| .
\end{eqnarray*}
This also leads to $  z\in \bigcup\limits_{i\in \bar{S},\atop j\in
S} \mathcal{K}_{i,j}(\mathcal {A}) \subseteq \mathcal{K}^S(\mathcal
{A}),$ and $\Omega^S(\mathcal {A}) \subseteq \mathcal{K}^S(\mathcal
{A})$.

If $r_p(\mathcal {A}) r_q^{\Delta^{S}}(\mathcal {A})>0$, then we can
equivalently express Inequality (\ref{neq2.3}) as
 \begin{equation} \label{neq2.4} \frac{|z-a_{q\cdots q}|-r_q^{\overline{\Delta^{S}}}(\mathcal {A})}{r_q^{\Delta^{S}}(\mathcal
{A})} \frac{|z-a_{p\cdots p}|}{ r_p(\mathcal {A})}\leq
1,\end{equation} which implies
 \begin{equation} \label{neq2.4.1} \frac{|z-a_{q\cdots q}|-r_q^{\overline{\Delta^{S}}}(\mathcal {A})}{r_q^{\Delta^{S}}(\mathcal
{A})}\leq 1,\end{equation} or
\begin{equation} \label{neq2.4.2} \frac{|z-a_{p\cdots p}|}{ r_p(\mathcal {A})}\leq1.\end{equation}
Let $a=|z-a_{t\cdots t}|$, $ b= r_q^{\overline{\Delta^{S}}}(\mathcal
{A})$, $ c=r_q^{\Delta^{S}}(\mathcal {A})- |a_{qp\cdots p}| $ and
$d=|a_{qp\cdots p}| $. When Inequality (\ref{neq2.4.1}) holds and
$d=|a_{qp\cdots p}| >0$,  from the part (I) in Lemma \ref{nle2.1}
and Inequality (\ref{neq2.4}) we have
\[ \frac{|z-a_{q\cdots q}|-r_q^p(\mathcal {A})}{|a_{qp\cdots p}|} \frac{|z-a_{p\cdots p}|}{ r_p(\mathcal {A})} \leq
 \frac{|z-a_{q\cdots q}|-r_q^{\overline{\Delta^{S}}}(\mathcal {A})}{r_q^{\Delta^{S}}(\mathcal
{A})} \frac{|z-a_{p\cdots p}|}{ r_p(\mathcal {A})}\leq 1\]
equivalently,
\[ |z-a_{p\cdots p}|\left(|z-a_{q\cdots q}|-r_q^p(\mathcal {A})\right) \leq r_p(\mathcal {A})|a_{qp\cdots p}|.\]
 This implies  $\Omega^S(\mathcal {A}) \subseteq
\mathcal{K}^S(\mathcal {A})$. When Inequality (\ref{neq2.4.1}) holds
and $d=|a_{qp\cdots p}| =0$, we easily get
\[ |z-a_{q\cdots q}|-r_q^p(\mathcal {A}) \leq 0=|a_{qp\cdots p}|.\] Hence,
\[ |z-a_{p\cdots p}|\left(|z-a_{q\cdots q}|-r_q^p(\mathcal {A})\right) \leq 0=r_p(\mathcal {A})|a_{qp\cdots p}|.\]
This also implies  $\Omega^S(\mathcal {A}) \subseteq
\mathcal{K}^S(\mathcal {A})$.  On the other hand, when Inequality
(\ref{neq2.4.2}) holds, we only need to prove  $\Omega^S(\mathcal
{A}) \subseteq \mathcal{K}^S(\mathcal {A})$ under the case that
\begin{equation}\label{neq2.9} \frac{|z-a_{q\cdots q}|-r_q^{\overline{\Delta^{S}}}(\mathcal {A})}{r_q^{\Delta^{S}}(\mathcal
{A})}> 1.\end{equation} Note that Inequality (\ref{neq2.9}) yields
\[ \frac{|z-a_{q\cdots q}| } {r_q(\mathcal {A})} >1.\] Hence, when
Inequality (\ref{neq2.4.2}) holds and $|a_{pq\cdots q}|>0$, we have
from Lemma \ref{nle2.1} and Inequality (\ref{neq2.4}) that
\[ \frac{|z-a_{q\cdots q}|}{r_q(\mathcal {A})} \frac{|z-a_{p\cdots p}|-r_p^q(\mathcal {A})}{ |a_{pq\cdots q}|} \leq
\frac{|z-a_{q\cdots q}|-r_q^{\overline{\Delta^{S}}}(\mathcal
{A})}{r_q^{\Delta^{S}}(\mathcal {A})} \frac{|z-a_{p\cdots p}|}{
r_p(\mathcal {A})}\leq 1\] equivalently,
\[ |z-a_{q\cdots q}|\left(|z-a_{p\cdots p}|-r_p^q(\mathcal {A})\right) \leq r_q(\mathcal {A})|a_{pq\cdots q}|.\]
This implies  $z\in \mathcal{K}_{p,q}(\mathcal {A}) \subseteq
\bigcup\limits_{i\in S ,\atop j\in \bar{S}}
\mathcal{K}_{i,j}(\mathcal {A}) \subseteq \mathcal{K}^S(\mathcal
{A})$ and $\Omega^S(\mathcal {A}) \subseteq \mathcal{K}^S(\mathcal
{A})$.
 And when Inequality (\ref{neq2.4.2}) holds and $|a_{pq\cdots q}|=0$, we easily get
\[ |z-a_{p\cdots p}|-r_p^q(\mathcal {A}) \leq 0=|a_{pq\cdots q }|.\] Hence,
\[ |z-a_{q\cdots q}|\left(|z-a_{p\cdots p}|-r_p^q(\mathcal {A})\right) \leq 0=r_q(\mathcal {A})|a_{pq\cdots q}|.\]
This also implies  $\Omega^S(\mathcal {A}) \subseteq
\mathcal{K}^S(\mathcal {A})$. \end{proof}

\begin{rmk}
For a complex tensor $\mathcal {A}\in C^{[m,n]}$ , $n\geq 2$,
 the set $\mathcal{K}^S(\mathcal {A}) $ consists
of $2|S|(n-|S|)$ sets $ \mathcal{K}_{i,j}(\mathcal {A})$, and the
set $\Omega^S(\mathcal {A})$ consists of $|S|(n-|S|)$ sets $
\Omega_{i,j}^S(\mathcal {A}) $ and $|S|(n-|S|)$ sets $
\Omega_{i,j}^{\bar{S}}(\mathcal {A}) $, where $S$ is a nonempty
proper subset of $N$. Hence, under the same computations,
$\Omega^S(\mathcal {A})$ captures all eigenvalues of $\mathcal {A}$
more precisely than $\mathcal{K}^S(\mathcal {A}) $.
\end{rmk}

\section{Sufficient conditions for positive (semi-)definiteness of tensors}
As applications of the results in Sections 2, we in this section
provide some checkable sufficient conditions for the positive
definiteness and positive semi-definiteness of tensors,
respectively. Before that, we give some definitions in
\cite{Di,Li0,Zh}.

\begin{definition}  \cite{Di,Zh} \label{def3.1}
A tensor $\mathcal {A}=(a_{i_1\cdots i_m})\in C^{[m,n]}$ is called a
(strictly) diagonally dominant tensor if for $i\in N$,
\begin{equation}\label{neq2.5} |a_{i\cdots i}|\geq (>) r_i(\mathcal{A}).\end{equation}
\end{definition}

\begin{definition}  \cite{Li0} \label{def3.2}
A tensor $\mathcal {A}=(a_{i_1\cdots i_m})\in C^{[m,n]}$ with $n\geq
2$ is called a quasi-doubly (strictly) diagonally dominant tensor if
for $i, j\in N$, $j\neq i$,
\begin{equation}\label{neq2.5.0} \left(|a_{i\cdots i}|- r_i^j(\mathcal{A}) \right)|a_{j\cdots j}| \geq (>) r_j(\mathcal{A}) |a_{ij\cdots j}|.\end{equation}
\end{definition}

\begin{definition} \label{def3.3}
Let $\mathcal {A}=(a_{i_1\cdots i_m})\in C^{[m,n]}$ with $n\geq 2$
and $S$ be a nonempty proper subset of $N$. $ \mathcal {A}$ is
called an $S$-$QDSDD_0$ ($S$-$QDSDD$) tensor if for each $i\in S$
and each $j\in \bar{S}$, Inequality (\ref{neq2.5.0}) holds and
\begin{equation}\label{neq2.5.1} \left(|a_{j\cdots j}|- r_j^i(\mathcal{A}) \right)|a_{i\cdots i}| \geq (>) r_i(\mathcal{A}) |a_{ji\cdots i}|.\end{equation}
\end{definition}

\begin{definition} \label{def3.4}
Let $\mathcal {A}=(a_{i_1\cdots i_m})\in C^{[m,n]}$ with $n\geq 2$
and $S$ be a nonempty proper subset of $N$. $ \mathcal {A}$ is
called an  $S$-$SDD_0$ ($S$-$SDD$) tensor if for each $i\in S$ and
each $j\in \bar{S}$, \begin{equation}\label{eq2.5.2}  |a_{i\cdots
i}| \left(|a_{j\cdots j}|-r_j^{\overline{\Delta^{S}}}(\mathcal
{A})\right)\geq (>) r_i(\mathcal {A}) r_j^{\Delta^{S}}(\mathcal
{A}),\end{equation} and
\begin{equation}\label{neq2.5.3} |a_{j\cdots j}|
\left(|a_{i\cdots i}|-r_i^{\overline{\Delta^{\bar{S}}}}(\mathcal
{A})\right)\geq (>) r_j(\mathcal {A})
r_i^{\Delta^{\bar{S}}}(\mathcal {A}).\end{equation}
\end{definition}

Next, we give the relationships between (strictly) diagonally
dominant tensors, quasi-doubly (strictly) diagonally dominant
tensors, $S$-$QDSDD_0$ ($S$-$QDSDD$) tensors and $S$-$SDD_0$
($S$-$SDD$) tensors.

\begin{proposition} \label{pro3.1}If  $\mathcal{A}=(a_{i_1\cdots i_m})\in C^{[m,n]}$
is a strictly diagonally dominant tensor, then $\mathcal{A}$ is a
quasi-doubly strictly diagonally dominant tensor. If $ \mathcal{A}$
is a diagonally dominant tensor, then $\mathcal{A}$ is a
quasi-doubly  diagonally dominant tensor.
\end{proposition}

\begin{proof}  If $\mathcal{A}$ is a strictly diagonally dominant tensor, then for any $i\in N$,
\[ |a_{i\cdots i}|> r_i(\mathcal{A}) ,\] equivalently,
\[|a_{i\cdots i}|-r_i^j(\mathcal{A}) > |a_{ij\cdots j}|.\]
Hence, for $i,j\in N$, $j\neq i$,
\[ |a_{i\cdots i}| > r_i(\mathcal{A}),\]
and
\[|a_{j\cdots j}|-r_j^i(\mathcal{A}) > |a_{ji\cdots i}|,\]
which implies that the strict inequality (\ref{neq2.5.0}) holds,
i.e., $\mathcal{A}$ is a quasi-doubly strictly diagonally dominant
tensor by  Definition \ref{def3.2}. Similarly, we can prove that if
$\mathcal{A}$ is a diagonally dominant tensor, then $\mathcal{A}$ is
a quasi-doubly diagonally dominant tensor. \end{proof}

\begin{proposition} \label{pro3.2}If  $\mathcal{A}=(a_{i_1\cdots i_m})\in C^{[m,n]}$
is a quasi-doubly strictly diagonally dominant tensor, then
$\mathcal{A}$ is an $S$-$QDSDD$ tensor. If $ \mathcal{A}$ is a
quasi-doubly diagonally dominant tensor, then $\mathcal{A}$ is an
$S$-$QDSDD_0$  tensor.
\end{proposition}

\begin{proof} If $\mathcal{A}$ is a quasi-doubly strictly diagonally dominant tensor, then for $i, j\in N$, $j\neq i$,
the strict inequality (\ref{neq2.5.0}) holds, i.e.,
\[\left(|a_{i\cdots i}|- r_i^j(\mathcal{A}) \right)|a_{j\cdots j}|
> r_j(\mathcal{A}) |a_{ij\cdots j}|.\]
For a given nonempty proper subset $S$ of $N$, we easily get that
for each $i\in S$ and each $j\in \bar{S}$,
\[\left(|a_{i\cdots i}|- r_i^j(\mathcal{A}) \right)|a_{j\cdots j}|
> r_j(\mathcal{A}) |a_{ij\cdots j}|,\] and
\[ \left(|a_{j\cdots j}|- r_j^i(\mathcal{A}) \right)|a_{i\cdots i}|
> r_i(\mathcal{A}) |a_{ji\cdots i}|.\]
Hence, $\mathcal{A}$ is an $S$-$QDSDD$ tensor. Similarly, we can
prove that  a quasi-doubly diagonally dominant tensor is an
$S$-$QDSDD_0$  tensor. \end{proof}

\begin{proposition} \label{pro3.3} Let $\mathcal{A}=(a_{i_1\cdots i_m})\in C^{[m,n]}$ and $S$ be a nonempty proper subset of $N$.
If  $\mathcal{A}$ is an $S$-$QDSDD$ tensor, then $\mathcal{A}$ is an
$S$-$SDD$  tensor. If $\mathcal{A}$ is an $S$-$QDSDD_0$  tensor,
then $\mathcal{A}$ is an $S$-$SDD_0$  tensor.
\end{proposition}

\begin{proof}  We only prove that an  $S$-$QDSDD$ tensor is an $S$-$SDD$ tensor,
and by a similar way, we can prove that an $S$-$QDSDD_0$ tensor is
an $S$-$SDD_0$ tensor.

 Let $\mathcal{A}$ be an  $S$-$QDSDD$ tensor. It is easy to see
from Definition  $\ref{def3.3}$ that either for any $i\in S$,
 \begin{equation} \label{eq3.4} |a_{i\cdots i}| >  r_i(\mathcal{A}),\end{equation}
or for any $j\in \bar{S}$, \begin{equation} \label{eq3.5}
|a_{j\cdots j}|
> r_j(\mathcal{A}). \end{equation} Without loss of generality, we
next suppose that for any $j\in \bar{S}$, Inequality (\ref{eq3.5})
holds. Hence, for any $j\in \bar{S}$,
\begin{equation} \label{eq3.6}
|a_{j\cdots j}|-r_j^i(\mathcal{A})
> |a_{ji\cdots i}| \end{equation}
and
\begin{equation} \label{eq3.7}
|a_{j\cdots j}|-r_j^{\overline{\Delta^{S}}}(\mathcal {A})
> r_j^{\Delta^{S}}(\mathcal
{A}). \end{equation}

Case I: for $i\in S$ such that Inequality (\ref{eq3.4} ) holds,
i.e.,
\begin{equation} \label{eq3.8}
|a_{i\cdots i}|-r_i^{\overline{\Delta^{\bar{S}}}}(\mathcal {A})
> r_i^{\Delta^{\bar{S}}}(\mathcal
{A}), \end{equation} by combining with Inequalities (\ref{eq3.5})
and (\ref{eq3.7}) we easily get that for this $i\in S$ and each
$j\in \bar{S}$, Inequalities (\ref{eq2.5.2}) and  (\ref{neq2.5.3})
hold.

Case II: for $i\in S$ such that
\[|a_{i\cdots i}| \leq  r_i(\mathcal{A}),\]
by Definition \ref{def3.3} we can get that $0<|a_{i\cdots i}| \leq
r_i(\mathcal{A})$,
\begin{equation} \label{eq3.9}
0< |a_{i\cdots i}|-r_i^j(\mathcal {A}) \leq |a_{ij\cdots j}|,
\end{equation}
and
\begin{equation} \label{eq3.10}
0<|a_{i\cdots i}|-r_i^{\overline{\Delta^{\bar{S}}}}(\mathcal {A})
\leq r_i^{\Delta^{\bar{S}}}(\mathcal {A}). \end{equation} Let
$a=|a_{i\cdots i}|$, $b=r_i^{\overline{\Delta^{\bar{S}}}}(\mathcal
{A})$, $c= r_i^{\Delta^{\bar{S}}}(\mathcal {A})-|a_{ij\cdots j}|$,
$d=|a_{ij\cdots j}|$, $e=|a_{j\cdots j}|$,
$f=r_j^{\overline{\Delta^{S}}}(\mathcal {A})$,
$g=r_j^{\Delta^{S}}(\mathcal {A})-|a_{ji\cdots i}|$ and
$h=|a_{ji\cdots i}|$. If $ r_j(\mathcal {A})\neq 0$ for some $j\in
\bar{S}$, then by Inequality (\ref{neq2.5.0}), Inequality
(\ref{eq3.9}) and by Lemma \ref{nle2.1}, we have
\[ \frac{|a_{i\cdots i}|- r_i^{\overline{\Delta^{\bar{S}}}}(\mathcal
{A})}{ r_i^{\Delta^{\bar{S}}}(\mathcal {A})} \frac{|a_{j\cdots
j}|}{r_j(\mathcal{A})} \geq \frac{|a_{i\cdots i}|-
r_i^j(\mathcal{A})}{ |a_{ij\cdots j}|} \frac{|a_{j\cdots
j}|}{r_j(\mathcal{A})}>1,\] and $ |a_{j\cdots j}| \left(|a_{i\cdots
i}|-r_i^{\overline{\Delta^{\bar{S}}}}(\mathcal {A})\right)>
r_j(\mathcal {A}) r_i^{\Delta^{\bar{S}}}(\mathcal {A}) $, i.e.,
Inequality (\ref{neq2.5.3}) holds. Similarly, by Inequality
(\ref{neq2.5.0}) and by Lemma \ref{nle2.1}, we can also get
\[ \frac{|a_{i\cdots i}|}{ r_i(\mathcal {A})}
\frac{|a_{j\cdots j}|-r_j^{\overline{\Delta^{S}}}(\mathcal
{A})}{r_j^{\Delta^{S}}(\mathcal {A})} \geq \frac{|a_{i\cdots i}|-
r_i^j(\mathcal{A})}{ |a_{ij\cdots j}|} \frac{|a_{j\cdots
j}|}{r_j(\mathcal{A})}>1,\] where $ \frac{|a_{j\cdots
j}|-r_j^{\overline{\Delta^{S}}}(\mathcal
{A})}{r_j^{\Delta^{S}}(\mathcal {A})}=+ \infty$ if
$r_j^{\Delta^{S}}(\mathcal {A})=0 $, and $ |a_{i\cdots i}|
\left(|a_{j\cdots j}|-r_j^{\overline{\Delta^{S}}}(\mathcal
{A})\right)> r_i(\mathcal {A}) r_j^{\Delta^{S}}(\mathcal {A}),$
i.e., Inequality  (\ref{eq2.5.2}) holds. On the other hand, if $
r_j(\mathcal {A})=0$ for some $j\in \bar{S}$, then $
r_j^{\overline{\Delta^{S}}}(\mathcal {A})=r_j^{\Delta^{S}}(\mathcal
{A})=0$. Obviously, Inequalities (\ref{eq2.5.2}) and
(\ref{neq2.5.3}) also hold. The conclusion follows from Cases I and
II.
\end{proof}

As shown in \cite{Li1,Li2}, by using eigenvalue localization sets
for tensors, one can give some corresponding checkable sufficient
conditions of the positive (semi-)definiteness of tensors. Here we
call a tensor $\mathcal {A}=(a_{i_1\cdots i_m})\in R^{[m,n]}$
symmetric \cite{Qi,Ya} if
\[a_{i_1\cdots i_m }= a_{\pi(i_1\cdots i_m )},\forall \pi\in
\Pi_m,\]where $\Pi_m$ is the permutation group of $m$ indices. And
an even-order real symmetric tensor is called positive
(semi-)definite, if its smallest H-eigenvalue is positive
(nonnegative). Next, a new checkable sufficient condition of the
positive (semi-)definiteness of tensors is obtained by using Theorem
\ref{nth2.1}.

\begin{thm}
\label{th3.1} Let $\mathcal {A}=(a_{i_1\cdots i_m})\in R^{[m,n]}$
with $n\geq 2$ and $S$ be a nonempty proper subset of $N$. If
$\mathcal {A}$ is an even-order symmetric $S$-$SDD$ ($S$-$SDD_0$)
tensor with $a_{k\cdots k}
> (\geq)  ~0$ for all $k\in N$,  then $\mathcal {A}$ is positive
(semi-)definite.
\end{thm}

\begin{proof} We need only prove that $\mathcal {A}$ is positive semi-definite,
 and by a similar way, we can prove that $\mathcal {A}$ is positive definite. Let $\lambda$ be an H-eigenvalue of $\mathcal {A}$.
Suppose on the contrary that $\lambda < 0$. From Theorem
\ref{nth2.1}, we have $\lambda\in \Omega^S (\mathcal {A})$  which
implies that there are $i_0,i_1\in S,j_0,j_1 \in \bar{S} $ such that
$\lambda\in \Omega_{i_0,j_0}^S(\mathcal {A})$ or $\lambda\in
\Omega_{j_1,i_1}^{\bar{S}}(\mathcal {A}) $, that is,
\[|\lambda-a_{i_0\cdots i_0}|
\left(|\lambda-a_{j_0\cdots
j_0}|-r_{j_0}^{\overline{\Delta^{S}}}(\mathcal {A})\right)\leq
r_{i_0}(\mathcal {A}) r_{j_0}^{\Delta^{S}}(\mathcal {A})\] or
\[|\lambda-a_{j_1\cdots j_1}|
\left(|\lambda-a_{i_1\cdots
i_1}|-r_{i_1}^{\overline{\Delta^{\bar{S}}}}(\mathcal {A})\right)\leq
r_{j_1}(\mathcal {A}) r_{i_1}^{\Delta^{\bar{S}}}(\mathcal {A}) .\]
On the other hand, since $\mathcal {A}$ is an $S$-$SDD_0$ tensor
with $a_{k\cdots k} \geq 0$ for all $k\in N$, we have that for
$i_0,i_1\in S,j_0,j_1 \in \bar{S} $, \[|\lambda-a_{i_0\cdots i_0}|
\left(|\lambda-a_{j_0\cdots
j_0}|-r_{j_0}^{\overline{\Delta^{S}}}(\mathcal {A})\right)>
|a_{i_0\cdots i_0}| \left(|a_{j_0\cdots
j_0}|-r_{j_0}^{\overline{\Delta^{S}}}(\mathcal {A})\right)  \geq
r_{i_0}(\mathcal {A}) r_{j_0}^{\Delta^{S}}(\mathcal {A})\] and
\[|\lambda-a_{j_1\cdots j_1}|
\left(|\lambda-a_{i_1\cdots
i_1}|-r_{i_1}^{\overline{\Delta^{\bar{S}}}}(\mathcal
{A})\right)>|a_{j_1\cdots j_1}| \left(|a_{i_1\cdots
i_1}|-r_{i_1}^{\overline{\Delta^{\bar{S}}}}(\mathcal {A})\right)
\geq r_{j_1}(\mathcal {A}) r_{i_1}^{\Delta^{\bar{S}}}(\mathcal {A})
.\] These lead to a contradiction. Hence, $\lambda\geq 0$, and
$\mathcal{A}$ is positive semi-definite. The conclusion follows.
\end{proof}

According to Theorem \ref{th3.1}, Proposition \ref{pro3.1},
Proposition \ref{pro3.2} and Proposition \ref{pro3.3}, we easily
obtain the following results which were also obtained in
\cite{Di,Li,Li1,Zh}.

\begin{corollary} \label{npro2.1} An even-order strictly diagonally dominant symmetric
tensor with all positive diagonal entries is positive definite. And
an even-order diagonally dominant symmetric tensor with all
nonnegative diagonal entries is positive semi-definite.
\end{corollary}

\begin{corollary} \label{npro2.2} An even-order
quasi-doubly strictly diagonally dominant symmetric tensor with all
positive diagonal entries is positive definite. And an even-order
quasi-doubly diagonally dominant tensor symmetric tensor with all
nonnegative diagonal entries is positive semi-definite.
\end{corollary}

\begin{corollary} \label{npro2.3} An even-order
$S$-$QDSDD$  tensor with all positive diagonal entries is positive
definite. And an even-order  $S$-$QDSDD_0$  symmetric tensor with
all nonnegative diagonal entries is positive semi-definite.
\end{corollary}

\begin{example}  {\rm  Let $\mathcal {A}=(a_{ijkl})\in R^{[4,3]}$ be a real symmetric tensor with elements
defined as follows: \[a_{1111} =5,a_{2222} =6,a_{3333} =3.3,
~a_{1112}=-0.1,~a_{1113}=0.1,~a_{1122}=-0.2,\]
\[a_{1123}=-0.2,a_{1133}=0,~a_{1222}=-0.1,~a_{1223}=0.3,~a_{1233}=0.1,\]
\[a_{1333}=-0.1,~a_{2223}=0.1,~a_{2233}=-0.1,~a_{2333}=0.2.\]
Let $S=\{1,2\}$. Obviously $\bar{S}=\{3\}$. By computations, we get
that
 \[|a_{1111}|\left(|a_{3333}|-r_3^1(\mathcal {A}) \right)
=-0.5000< 0.3800=|a_{3111}| r_1(\mathcal {A}).\] Hence, $\mathcal
{A}$ is not an  $S$-$QDSDD_0$ tensor, consequently, not a strictly
diagonally dominant symmetric tensor or a quasi-doubly strictly
diagonally dominant tensor, and hence we can not use Corollary \ref
{npro2.1}, Corollary \ref {npro2.2} or Corollary \ref {npro2.3} to
determine the positiveness of $\mathcal {A}$. However, it is easy to
get that
\[|a_{1111}|\left( |a_{3333}|-r_3^{\overline{\Delta^{S}}}(\mathcal {A})
\right)=4.0000>3.8000=r_1(\mathcal{A})r_3^{\Delta^{S}}(\mathcal
{A}),\]
\[|a_{3333}|\left( |a_{1111}|-r_1^{\overline{\Delta^{\bar{S}}}}(\mathcal {A})
\right)=4.2900>0.3500=r_3(\mathcal{A})r_1^{\Delta^{\bar{S}}}(\mathcal{A}),\]

\[|a_{2222}|\left( |a_{3333}|-r_3^{\overline{\Delta^{S}}}(\mathcal {A})
\right)=4.8000>4.5000=r_2(\mathcal{A})r_3^{\Delta^{S}}(\mathcal{A})\]
and
\[|a_{3333}|\left( |a_{2222}|-r_2^{\overline{\Delta^{\bar{S}}}}(\mathcal {A})
\right)=5.6100>0.7000=r_3(\mathcal{A})r_2^{\Delta^{\bar{S}}}(\mathcal{A}),\]
i.e., $A$ is an  $S$-$SDD$ tensor. By Theorem \ref{th3.1}}, $A$ is
positive definite. \end{example}

\section*{Acknowledgements}
This work is supported by National Natural Science Foundations of
China (11361074), Natural Science Foundations of Yunnan Province
(2013FD002) and IRTSTYN.





\begin{thebibliography}{10}

\bibitem{brauer}
A. Brauer,  Limits for the characteristic roots of a matrix II, Duke
Math. J. 14 (1947) 21-26.

\bibitem{Bo}
N.K. Bose, A.R. Modaress, General procedure for multivariable
polynomial positivity with control applications, IEEE Trans.
Automat. Control AC21 (1976) 596-601.



\bibitem{Ch}
Y. Chen, Y. Dai, D. Han, W. Sun, Positive semidefinite generalized
diffusion tensor imaging via quadratic semidefinite programming,
SIAM J. Imaging Sci. 6 (2013) 1531-1552.

\bibitem{Ch1}
K. C. Chang,  K. Pearson,  T. Zhang, Perron-Frobenius theorem for
nonnegative tensors, Commun. Math. Sci. 6 (2008) 507-520.

\bibitem{Di}
W. Ding, L. Qi, Y. Wei, $M$-tensors and nonsingular $M$-tensors,
Linear Algebra Appl. 439 (2013) 3264-3278.


\bibitem{Ger}
S. Ger\v{s}gorin,  $\ddot{U}$ber die Abgrenzung der Eigenwerte einer
Matrix. \emph{Izv. Akad. Nauk SSSR Ser. Mat.} 1 (1931) 749-754.

\bibitem{Hu}
S. Hu, Z. Huang, H. Ni, L. Qi, Positive definiteness of diffusion
kurtosis imaging, Inverse Probl. Imaging 6 (2012) 57-75.



\bibitem{Ju}
E.I. Jury, M. Mansour, Positivity and nonnegativity conditions of a
quartic equation and related problems, IEEE Trans. Automat. Control
AC26 (1981) 444-451.


\bibitem{Ko1}
T.G. Kolda,  J.R. Mayo, Shifted power method for computing tensor
eigenpairs. SIAM Journal on Matrix Analysis and Applications, 32(4)
(2011) 1095-1124.


\bibitem{Ku}
W.H. Ku, Explicit criterion for the positive definiteness of a
general quartic form, IEEE Trans. Automat. Control AC10 (1965)
372-373.

\bibitem{Li}
C.Q. Li, F. Wang, J.X. Zhao, Y. Zhu, Y.T. Li, Criterions for the
positive definiteness of real supersymmetric tensors, Journal of
Computational and Applied Mathematics 255 (2014) 1-14.

\bibitem{Li0}
C.Q. Li, Y.T. Li, Double $B$-tensors and quasi-double $B$-tensors,
Linear Algebra and its Applications 466 (2015) 343-356.

\bibitem{Li1}
C.Q. Li, Y.T. Li,  X. Kong, New eigenvalue inclusion sets for
tensors, Numer. Linear Algebra Appl. 21 (2014) 39-50.


\bibitem{Li2}
C.Q. Li, Y.T. Li, An eigenvalue localization set for tensors with
applications to determine the positive (semi-)definiteness of
tensors, Linear and Multilinear Algebra,
DOI:10.1080/03081087.2015.1049582.

\bibitem{Li3}
C.Q. Li, Z. Chen, Y.T. Li, A new eigenvalue inclusion set for
tensors and its applications, Linear Algebra and its Applications
481 (2015) 36-53.


\bibitem{Lim}
L.H. Lim, Singular values and eigenvalues of tensors: A variational
approach. in CAMSAP'05: Proceeding of the IEEE International
Workshop on Computational Advances in MultiSensor Adaptive
Processing 2005; 129-132.

\bibitem{Liu}
Y. Liu,  G. Zhou, N.F. Ibrahim, An always convergent algorithm for
the largest eigenvalue of an irreducible nonnegative tensor. Journal
of Computational and Applied Mathematics  235 (2010)  286-292.

\bibitem{Ng}
M. Ng, L. Qi, G. Zhou,  Finding the largest eigenvalue of a
nonnegative tensor. SIAM J. Matrix Anal. Appl.  31 (2009) 1090-1099.

\bibitem{Ni}
Q. Ni, L. Qi,  F. Wang,  An eigenvalue method for the positive
definiteness identification problem. IEEE Transactions on Automatic
Control 53 (2008) 1096-1107.

\bibitem{Qi}
L. Qi, Eigenvalues of a real supersymmetric tensor, J. Symbolic
Comput. 40 (2005) 1302-1324.

\bibitem{Qi1}
L. Qi, $H^+$-eigenvalues of Laplacian and signless Laplacian
tensors, Commun. Math. Sci. 12 (2014) 1045-1064.



\bibitem{Qi5}
L. Qi, G. Yu, E.X. Wu, Higher order positive semi-definite diffusion
tensor imaging, SIAM J. Imaging Sci. 3 (2010) 416-433.


\bibitem{Va}
R.S. Varga, Ger\v{s}gorin and his circles, Springer-Verlag, Berlin,
2004.

\bibitem{Wa}
Y. Wang, L. Qi, X. Zhang, A practical method for computing the
largest M-eigenvalue of a fourth-order partially symmetric tensor.
Numerical Linear Algebra with Applications 16 (2009) 589-601.

\bibitem{Wa1}
F. Wang, L. Qi, Comments on ¡®Explicit criterion for the positive
definiteness of a general quartic form¡¯, IEEE Trans. Automat.
Control 50 (2005) 416-418.


\bibitem{Ya}
Y. Yang,  Q. Yang, Further results for Perron-Frobenius Theorem for
nonnegative tensors, SIAM. J. Matrix Anal. Appl. 31 (2010)
2517-2530.


\bibitem{Ya1}
Q Yang,Y. Yang,  Further results for Perron-Frobenius theorem for
nonnegative tensors II. SIAM. J. Matrix Anal. Appl.  32 (2011)
1236-1250.


\bibitem{Zh}
L. Zhang, L. Qi, G. Zhou, $M$-tensors and some applications, SIAM J.
Matrix Anal. Appl. 35 (2014) 437-452.

\end{thebibliography}



\end{document}